\documentclass{amsart}
\usepackage{amsmath, amssymb, latexsym,  amsthm, amscd}

\newcommand{ \PP }{{\mathbb P}}
\newcommand{ \CC }{{\mathbb C}}

\newtheorem{thm}{Theorem}[section]
 \newtheorem{cor}[thm]{Corollary}
 \newtheorem{lem}[thm]{Lemma}

 \theoremstyle{definition}
 \newtheorem{defn}[thm]{Definition}
 \theoremstyle{remark}
 \newtheorem{rem}[thm]{Remark}

\newcommand\map{\dasharrow}

\begin{document}

\title{One example of general unidentifiable tensors}

\author[L. Chiantini]{Luca Chiantini}
\address[L. Chiantini]{Dipartimento di Ingegneria dell'Informazione e Scienze Matematiche 
\\ Universit\`{a} di Siena \\ Pian dei Mantellini 44 \\ 53100 Siena,  Italia}
\email{luca.chiantini@unisi.it}
\urladdr{\tt{http://www.mat.unisi.it/newsito/docente.php?id=4}}

\author[M. Mella]{Massimiliano Mella}
\address[M. Mella]{ Dipartimento di  Matematica e Informatica\\ Universit\`{a} di Ferrara\\ 
Via Machiavelli 35\\ 44121 Ferrara, Italia}
\email{mll@unife.it}
\urladdr{{\tt{http://docente.unife.it/massimiliano.mella}}}

\author[G. Ottaviani]{Giorgio Ottaviani}
\address[G. Ottaviani]{Dipartimento di Matematica e Informatica 'Ulisse Dini'\\ 
Universit\`{a} di Firenze \\ Viale Morgagni 67/A \\ 50134 Firenze, Italia}
\email{ottavian@math.unifi.it}
\urladdr{\tt{http://web.math.unifi.it/users/ottavian/}}

\date{\today}

\begin{abstract} {The identifiability  of parameters in a probabilistic model is
a crucial notion in statistical inference. } We prove that a general tensor of rank $8$ in
$\CC^3\otimes\CC^6\otimes\CC^6$ has at least $6$ decompositions
as sum of simple tensors, { so it is not $8$-identifiable}. This is the highest known example
of balanced tensors of {dimension} $3$, which are not $k$-identifiable,
when $k$ is smaller than the generic rank.
\end{abstract}


 \maketitle

\thispagestyle{empty}

\section{Introduction}

 The decomposition of tensors $T\in \CC^{a_1+1}\otimes\cdots\otimes\CC^{a_q+1}$ as a sum
 of {\it simple} tensors (i.e. tensors of rank $1$) is a central
 problem for many applications of Multilinar Algebra to Algebraic
 Statistics, signal theory, coding theory and others.

 {For  statistical inference, it is meaningful  to know if  
 a probability distribution, arising from a
model, uniquely determines the parameters that produced it. 
When this happens, the parameters are called {\it identifiable}. 
The notion of {\it generic} identifiability for parametric models  
has been considered in \cite{AMR09} and in \cite{SR12} \S 2.2.
Indeed, conditions which guarantee the uniqueness of this decomposition, 
for generic tensors in the model, are quite important in the applications. 
When generic identifiability holds, the set of non-identifiable
parameters has measure zero, thus parameter inference is still meaningful.
Notice that many decomposition algorithms converge to {\it one} decomposition,
hence a uniqueness result guarantees that the decomposition found is the chased one. 
We refer to \cite{KB} and  its huge  reference list, for more details.}

Even from a purely theoretical point of view, the study of the decomposition 
shows some beautiful link between Multilinear Algebra and Projective Geometry.

The present paper is devoted to study one intriguing special case,
which shows an exceptional behavior.

Among tensors $T\in \CC^{a_1+1}\otimes\cdots\otimes\CC^{a_q+1}$ 
whose rank has the generic value,
 only one example is known when we have identifiability, that is  $q=3$, 
{$a_1=1$, $a_2=a_3$.}
On the contrary, general tensors whose rank is
{\it smaller} than the generic value, often have a unique decomposition.

Excluding the cases of matrices (tensors of {dimension} $q=2$), 
identifiability is known to hold when the rank $k$ is {\it small}.
An evidence is given for $q=3$ by the celebrated Kruskal's bound \cite{K},
which, for general tensors of given rank, is refined and extended
in a series of papers (see Strassen's paper \cite{Str}, the recent paper \cite{CO}).

Let's order the $a_i$'s so that $a_1\leq \dots\leq a_q$.
In \cite{BCO} Corollary 8.4 it was proved that, with the assumption $a_q\ge \prod_{i=1}^{q-1}(a_i+1)-\left(\sum_{i=1}^{q-1}a_i\right)$,
the variety $\PP^{a_1}\times\ldots\times\PP^{a_{q}}$ is $k$-identifiable if and only if
$$k\le \prod_{i=1}^{q-1}(a_i+1)-\left(1+\sum_{i=1}^{q-1}a_i\right)$$ 
The general tensor of rank $k>\prod_{i=1}^{q-1}(a_i+1)-\left(1+\sum_{i=1}^{q-1}a_i\right)$
has not a unique decomposition.
After this result, we say that a tensor is {\it unbalanced}
if $a_q\ge \prod_{i=1}^{q-1}(a_i+1)-\left(\sum_{i=1}^{q-1}a_i\right)$ .
This range is one unity larger than the corresponding unbalanced range considered in \cite{AOP}
while studying the dimension of secant varieties of Segre varieties (see \S 8 of \cite{BCO} where these two notions were compared).

 In the case $a_q\le  \prod_{i=1}^{q-1}(a_i+1)-\left(1+\sum_{i=1}^{q-1}a_i\right)$
the corresponding tensors are called {\it balanced}.

Only few examples of balanced tensors, whose rank is smaller than the generic value,   
are known to be not generically identifiable. We mention the case 
of tensors of rank $5$ in $(\CC^2)^{\otimes 5}$ (\cite{BC}) and, in {dimension} $3$,
tensors of rank $3$ in $(\CC^3)^{\otimes 3}$ (classical, see \cite{Str} \S 4)
and tensors of rank $6$ in $(\CC^4)^{\otimes 3}$  (\cite{CO}, Theorem 1.3).

A computer aided analysis (see \cite{BCO} Theor. 7.5) shows that when the numbers $a_i$'s grow, 
sporadic examples disappear, and we expect that a general balanced tensor, 
of rank smaller than the generic value, is  identifiable.

The present paper is devoted to illustrate one sporadic example,
which we believe should  be the last one, for balanced tensors 
of {dimension} $q=3$. Namely, we use a  geometric approach to show
that general tensors of rank $8$ in $\CC^3\otimes \CC^6\otimes\CC^6$
are not uniquely decomposable. Notice that tensors of the mentioned type
have generic rank equal to $9$.

The proof of the non-uniqueness is based on the weak-defectivity principle,
classically introduced by Terracini (\cite{Terr}).
We refer to \cite{CC,M}  and the introduction of \cite{BC}
for an account of the geometric reduction of the problem.

In details, we prove that through $8$ general points of the Segre
variety $\PP^2\times\PP^5\times\PP^5$, which corresponds to
simple tensors in $\CC^3\otimes \CC^6\otimes\CC^6$, one can find
a special fourfold $Y$ which is the Segre-Veronese image
of $\PP^2\times\PP^1\times\PP^1$, embedded by forms of type $(3,1,1)$.
Since through a general point of the span $\PP^{39}$ of $Y$
one can find many linear $7$-spaces which are $8$-secant to $Y$,
then by \cite{CC} Theorem 2.9, it follows the weak defectivity and the
non identifiability of our tensors.  

The example is interesting also because the subvariety $Y$, which produces
the non-identifiability of tensors of rank $8$ in $\CC^3\otimes \CC^6\otimes\CC^6$,
is quite complicate. In particular, we are unable to estimate
how many $8$-secant spaces to $Y$ are there through a general point of 
the span $\PP^{39}$. Consequently, we are unable to determine
how many different decomposition are there, for a general tensor $T$ as above.
We simply know that the number is finite, and at least $6$.

Let us mention that, from the geometrical point of view, the existence
of the subvariety $Y$ through $8$ general points of  $\PP^2\times\PP^5\times\PP^5$
is proved by some ''ad hoc'' argument. A complete theory of special
subvarieties that one can find through general points of Segre
varieties, seems actually far beyond our reach.

\section{Preliminaries}
For basic facts about the geometric point of view 
on tensors we follow \cite{Land}.

Given any irreducible projective variety $X$, we denote 
by $S_k(X)$ the $k$-th secant variety of $X$, that is the Zariski closure 
of the set $\bigcup_{x_1,\ldots, x_k\in X}<x_1,\ldots, x_k>$. 
$S_k(X)$ is indeed the Zariski closure of the set of 
elements having $X$-rank equal to $k$.

In the space $\PP^N=\PP(\CC^{a_1+1}\otimes\cdots\otimes\CC^{a_q+1})$,
where 
{$N=-1+\prod_{i=1}^q (a_q+1)$}, the (projectification of the)
cone $X$ of simple tensors corresponds to the embedding
of $\PP^{a_1}\times\cdots\times\PP^{a_q}$, via the Segre map.
The (projectification of the) cone of tensors of rank $k$ is an open dense
subset of the secant variety $S_k(X)$.

We recall from \cite{CO} def. 2.1 the following

\begin{defn} $X$ is called $k$-identifiable if the general element 
of  $S_k(X)$ has a unique expression
as the sum of $k$ elements of $X$.
\end{defn}

Thus, in our notation, we say that $\PP^{a_1}\times\cdots\times\PP^{a_q}$ is $k$-identifiable
if the general tensors in  $\CC^{a_1+1}\otimes\cdots\otimes\CC^{a_q+1}$ of rank $k$
has a unique decomposition as a sum of simple tensors.
\smallskip

A complete list  of known Segre varieties $X=\PP(\CC^{a_1+1})\times\PP(\CC^{a_2+1})\times\PP(\CC^{a_3+1})$, 
with $1\leq a_1\leq a_2\leq a_3\leq 6$, for which a computer based algorithm 
does not prove the $k$-identifiability, is provided in \cite{CO}, \S 5, see also \cite{BCO} \S 7.
The list corresponds to the case of tensors of {dimension} $3$, for which
the algorithm cannot prove the uniqueness of the decomposition.

In all the examples, except for two of them, it is indeed well known
that general tensors of rank $k$ have infinitely many decompositions. 

The two remaining cases are listed below:

{\small
$$\begin{array}{lccl}
(a_1,a_2,a_3)&k\\
\\
\hline\\
(3,3,3)&6 \\
\textrm{}&&& \\
\hline\\ 
(2,5,5)&8 
\end{array}$$}

In the first case, the effective proof that $X$ is not $6$-identifiable
(and the general tensor of rank $6$ has exactly $2$ decompositions)
is contained in \cite{CO}, Theorem 1.3.

The latter case needs an ''ad hoc'' analysis which  
is the target of the present note.

\smallskip

Our main tool is to prove the existence of particular,
very degenerate subvarieties $Y$, through $k$ general points 
of the Segre variety $X=\PP(\CC^{a_1+1})\times\PP(\CC^{a_2+1})\times\PP(\CC^{a_3+1})$.

Indeed, we recall the following:

\begin{thm} \label{inverseweak} Let $X$ be a  projective,
irreducible non--degenerate variety of dimension $n$ in
$\PP^r$, $r>nk+k-1$. Suppose that for any general $k$-tuples
of points $x_1,\dots,x_k \in  X$ one can find a subvariety
$Y$ of  pure  dimension $m>0$  containing the points
$x_1,\dots,x_k $, whose span has dimension
$$ \dim(\langle Y\rangle) = km+k-1.$$
Assume that $S_k(Y)=\langle W\rangle$ and moreover assume that 
through a general point of $\langle Y\rangle$ one finds $\mu_k>1$
$k$-secant $(k-1)$-linear spaces.

Then $X$ is not $k$-identifiable. Indeed through a general point 
of $S_k(X)$ one finds at least $\mu_k$ $k$-secant $(k-1)$-linear spaces.
\end{thm}
\begin{proof} It is essentially Theorem 2.9 of \cite{CC}.
\end{proof}

\section{Verifying the unidentifiability}\label{mainlem}

From this point on, we focus our attention to the vector space $V$
of tensors of type $\CC^3\otimes\CC^6\otimes\CC^6$, which has dimension $108$.
From the projective point of view, simple tensors in $V$ corresponds to points
of  the Segre embedding of $\PP^2\times\PP^5\times\PP^5$ into 
$\PP(\CC^3\otimes\CC^6\otimes\CC^6)=\PP^{107}$.

We also fix the rank $k=8$, i.e. we consider the eighth secant variety $S_8(X)$.
We know that $X$ is not $8$-defective, so that $S_8(X)$ has projective
dimension $103$ (see \cite{CO}, \S 5). This means that the subvariety (cone) of tensors of  
rank $8$ in $V$ has the expected dimension $104$.

Using a computer-based calculation, the guess is that $X$ is $8$-weakly
defective, with a contact variety of dimension $4$ and degree $108$.

In order to verify the guess, we need a series of lemmas.

\begin{lem}\label{P1P2} Fix eight general points $P_1,\dots,P_8$ of $\PP^5$
and fix eight general points $Q_1,\dots,Q_8$ of $\PP^2$. Then there exists
a Segre embedding $s: \PP^2\times\PP^1\to \PP^5$ for which the line
$s(\{Q_i\}\times\PP^1)$ contains $P_i$ for all $i$. In other words, each $P_i$
lies in $s(\PP^2\times \PP^1)$ and $\pi o s^{-1}(P_i)=Q_i$, where
$\pi$ is the projection $\PP^2\times\PP^1\to \PP^2$.
\end{lem}
\begin{proof} The embeddings $\PP^2\times\PP^1\to \PP^5$ are parametrized by
the quotient group $G= Aut(\PP^5)/(Aut(\PP^2)\times Aut(\PP^1))$, which has dimension $24$.
Since the eight points $P_i$ are general, we have an $8$-dimensional family $\mathcal S$ 
of embeddings $s$ for which $P_1,\dots,P_8\in s(\PP^2\times\PP^1)$. 
Since the unique automorphism of $\PP^2$ which fixes four general points is the identity,
as $s$ varies in $\mathcal S$, the family of $4$-tuples
$(\pi o s^{-1}(P_1),\dots,\pi o s^{-1}(P_4))$ dominates $(\PP^2)^4$.
Since the group $Aut(\PP^2)$ acts transitively on the points 
$\pi o s^{-1}(P_5),\dots,\pi o s^{-1}(P_8)$, it follows that the orbit
of the set $\{\pi o s^{-1}(P_5),\dots,\pi o s^{-1}(P_8)\}$, under
 $G\times Aut(\PP^2)$, dominates $(\PP^2)^8$. The claim follows.  
\end{proof}

\begin{lem}\label{P1P1P2} Through $8$ general points $x_1,\dots,x_8$
 of $X$ one can find a fourfold $Y$
which corresponds to the embedding of $\PP^2\times\PP^1\times\PP^1$ into $\PP^{39}$,
mapped by divisors of multidegree $(3,1,1)$.
\end{lem}
\begin{proof} We send $\PP^2\times\PP^1\times\PP^1$ to the three factors
$\PP^2$, $\PP^5$ and $\PP^5$, by using  the identity on $\PP^2$ and divisors
$D=(1,1,0)$ and $D'=(1,0,1)$  respectively. Thus, we need to prove
that we can arrange this map $\zeta$ so that the image passes through 
$8$ general points of $X$.

The choice of eight general points in $\PP^2\times\PP^5\times\PP^5$ corresponds
to the choice of $8$ general points in each factor. By the previous Lemma, 
for a general choice of points $Q_1,\dots,Q_8\in\PP^2$, $P_1,\dots,P_8
\in \PP^5$ and $P'_1,\dots,P'_8 \in \PP^5$, we can find divisors $D,D'$,
which define Segre embeddings $s,s'$ of $\PP^2\times\PP^1$ into $\PP^5$, for which 
each $P_i$ (resp. each $P'_i$) lies in the line $s(\{Q_i\}\times\PP^1)$
(resp. $s'(\{Q_i\}\times\PP^1)$).

It follows that $Y=\zeta(\PP^2\times\PP^1\times \PP^1)$ passes through
each point $x_i$, $i=1,\dots,8$.
\end{proof}

The following Lemma would be easy, provided one knows a table of $4$-dimensional
varieties in $\PP^{39}$, whose $8$-th secant order is different from $1$.
Since the table is missing, we need to compute directly what happens
for the Segre product $\PP^2\times\PP^1\times\PP^1$.

\begin{lem} \label{noOASS} Let $Y$ be an embedding of $\PP^2\times\PP^1\times\PP^1$ 
into $\PP^{39}$, through a divisor of type $(3,1,1)$. 
Then through a general point $y\in\PP^{39}$ one can draw
at least $6$ spaces of dimension $7$,  which are $8$-secant to $Y$.
\end{lem}
\begin{proof} We consider the tangential projection from the tangent spaces at $7$ general points $\{y_1,\ldots,y_7\}$
of $Y$, which is a rational map $\tau_{y_1,\ldots,y_7}\map Y\to\PP^{4}$.
By the Theorem 4.2 (vi) of \cite{CR} we have that the number of seven dimensional spaces
which are $8$-secant to $Y$ and contain a general point $y\in\PP^{39}$ is $\ge \deg\tau_{y_1,\ldots,y_7}$, for a general choice of points $\{y_1,\ldots,y_7\}$.
So it is enough to show that $\deg\tau_{y_1,\ldots,y_7}=6$.

A computer based algorithm, implemented in M2 \cite{GS}, which is available in the ancillary files of the 
arXiv submission of this paper,
shows that there exists a 7-uple $\{\overline{y}_1,\ldots, \overline{y}_7\}$ and a point $\overline{p}\in\PP^4$ such that
the fiber $\tau_{\overline{y}_1,\ldots,\overline{y}_7}^{-1}(\overline{p})$ consists of $6$ reduced points.

Consider the rational map
$$\tau:(Y)^7\times\PP^{39}\map (Y)^7\times\PP^4$$
induced by the tangential projection. That is
$$\tau(y_1,\ldots,y_7,p)=(y_1,\ldots,y_7,\tau_{y_1,\ldots,y_7}(p)).$$
Then, after resolving the indeterminacy of the map $\tau$, 
 we get, from the Stein factorization, that the general fiber of $\tau$ consists of six points,
so that $\deg\tau_{y_1,\ldots,y_7}=6$.
\end{proof}

\begin{rem} Computer experiments show that the base locus of $\tau_{Y,s}$ consists
of $2s$ lines for $s\le 6$ (each tangent space at a point $y\in Y$ meets $Y$ in two lines)
and consists of $14$ lines plus $4$ extra points for $s=7$. We do not know how to prove theoretically
the existence of these $4$ points in the base locus.
\end{rem}

Now we can use the approach of \cite{CC} to prove that $X$ is not $8$-identifiable.

\begin{thm}\label{xnot8} $X$ is not $8$-identifiable. Through a general point 
$Q\in\PP^{107}$ one can draw
at least $6$ spaces of dimension $7$,  which are $8$-secant to $X$.
\end{thm}
\begin{proof} Fix $8$ general points $P_1,\dots,P_8\in X$ and a general point 
$Q\in\langle P_1,\dots,P_8\rangle$, so that $Q$ is a general point of the $8$-th secant 
variety of $X$. By Lemma \ref{P1P1P2}, the eight points are contained in the image 
$Y\subset X$ of  a Segre-Veronese embedding  of $\PP^2\times\PP^1\times\PP^1$ 
through a divisor of type $(3,1,1)$.
$Y$ spans a $\PP^{39}$, which clearly contains $Q$, and $Q$ is a general point of $\PP^{39}$. 
By Lemma \ref{noOASS}, one finds $6$ linear spaces of (projective)
dimension $7$, which are $8$-secant
to $Y$ and contain $Q$. Since these spaces are also $8$-secant to $X$, the claim follows.
\end{proof}

From a geometric point of view, Theorem 2.4 of \cite{CC} implies the following.

\begin{cor} $X$ is $8$-weakly defective. A general hyperplane
which is tangent to $X$ at $8$ general point, is also tangent along a subvariety
$Y$ of dimension $4$, described above.
\end{cor}

\begin{rem} One would like to conclude that through a general point
of the $8$-secant variety of $X$ one can find exactly $6$ spaces of dimension $7$,
which are $8$-secant to $X$.

In other words, one would like to conclude that a general tensor of type $(3,6,6)$
and rank $8$ can be written as a sum of $8$ decomposable tensors in exactly $6$ ways. 

Unfortunately, we can only conclude that there are {\it at least} $6$ decompositions.
One reason is that the lower bound with $\deg\tau_{Y,7}$ 
considered in the proof can be strict. For example,
if $Y$ is the $8$-Veronese embedding of $\PP^2$, 
then Ranestad and Schreyer prove (see Theorem 1.7 (iv) of \cite{RS})
that a general polynomial of degree $8$ has exactly $16$ decompositions
as the sum of $15$ powers of linear forms. On the other hand,
 the tangential projection from $14$ points has base locus given by
the $14$ points themselves and so its degree is $8^2-14\cdot 4=8<16$.

Moreover, there could be more than one Segre-Veronese variety like $Y$,
passing through $8$ general points of $X$.

\end{rem} 

Using Terracini's interpretation of the secant varieties of Segre varieties
(\cite{Terr}) and the trick of \cite{BBCC}, we can translate the main Theorem
into a theorem about linear systems of matrices.

\begin{cor} Let $\mathcal M$ be a linear system of
$5\times 5$ matrices, with (affine) dimension $3$.
Assume that $\mathcal M$ has rank $8$, i.e. there are $8$
matrices of rank $1$ which generate all the elements of $\mathcal M$.
Then there are at least $6$ sets of $8$ rank $1$ matrices, whose spans
contain $\mathcal M$.
\end{cor}
\begin{proof} It is a straightforward consequence of the main Theorem of \cite{BBCC},
see Remark 4.2 iii) there.
\end{proof}

\end{document}